\documentclass{birkjour}
\usepackage[utf8]{inputenc}
\usepackage{amsmath}
\usepackage{amsthm}
\usepackage{amsfonts}
\usepackage{amssymb,multicol}
\usepackage{graphicx}
\usepackage{ytableau}

\newtheorem{theorem}{Theorem}
\numberwithin{theorem}{section}
\newtheorem{lemma}[theorem]{Lemma}

\newtheorem{example}{Example}
\numberwithin{example}{section}
\title[False theta function]{A bijective proof of a false theta function identity from Ramanujan’s lost notebook}
\author{Hannah Burson}
\address{%
Department of Mathematics, University of Illinois \\
1409 W. Green St.\\
Urbana, IL 61801}
\email{hburso2@illinois.edu} 

\begin{document}

\begin{abstract}
In his lost notebook, Ramanujan listed 5 identities related to the false theta function $$f(q)=\sum_{n=0}^\infty (-1)^nq^{n(n+1)/2}.$$ 
A new combinatorial interpretation and proof of one of these identities is given. The methods of the proof allow for new multivariate generalizations of this identity. Additionally, the same technique can be used to obtained a combinatorial interpretation of another one of the identities. 
\end{abstract}


\maketitle
\section{Introduction}
L. J Rogers \cite{Rogers1917} introduced false theta functions, which are series that would be classical theta functions except for changes in signs of an infinite number of terms. In Ramanujan's notebooks \cite{RamanujanNB} and lost notebook \cite{RamanujanLNB}, he recorded many false theta function identities that he discovered.  However, in Ramanujan's last letter to Hardy in 1920, Ramanujan introduced mock theta functions and shifted his focus away from false theta functions.  The mathematical community followed Ramanujan's lead and largely ignored false theta functions for the next several decades.  

In recent times, there has been an increase in interest in false theta functions. G.E. Andrews devoted a section of \cite{Andrews1979} to partition theoretic applications of false theta functions. More recently, such as in \cite{Alladi2009}, \cite{Alladi2010}, \cite{Berndt2010}, \cite{Berndt2003}, and \cite{Kim2010}, researchers have found combinatorial proofs of some of Ramanujan's false theta function identities.  

In his lost notebook \cite{RamanujanLNB} (c.f. \cite[p.~227]{Andrews2005}), Ramanujan stated five identities related to the false theta function
\begin{equation}\label{fq}
f(q)=\sum_{n=0}^\infty (-1)^nq^{n(n+1)/2}, \quad |q|<1.
\end{equation} 

These identities were first proved by Andrews in \cite{Andrews1981}, using identities such as the Rogers-Fine identity and Heine's transformation. Other analytic proofs have been given in \cite{Andrews2007}, \cite{Chu2010}, and \cite{Wang2018}. There are no previously known bijective proofs of any of these identities.  

In this paper, we provide a combinatorial proof of one of Ramanujan's identities for false theta functions. We adopt the standard $q$-series notation  \begin{align*}
(a;q)_n=\prod_{j=0}^{n-1}(1-aq^j), \quad |q|<1, n\in\{0,1,2\ldots\}.
\end{align*}
\begin{theorem}[Ramanujan]\label{fq4} 
If $f(q)$ is defined by (\ref{fq}), then for $|q|<1$,
$$\sum_{n=0}^\infty \frac{(q;q^2)_n\, q^n}{(-q;q^2)_{n+1}}=f(q^4).$$
\end{theorem}

In this paper, we will focus on a bijective proof of Theorem \ref{fq4}. In Section 2, we explain the necessary background on partitions.  Then, in Section 3, we introduce a new combinatorial analogue of Theorem \ref{fq4} and give its bijective proof in Section 4. In Section 5, we introduce new identities that arise from generalizing the proof in Section 4. Finally, in Section 6, we give a similar combinatorial interpretation of another one of Ramanujan's identities.

\section{Background}
We use several tools from the theory of partitions. Recall that a partition of $n$ is a non-increasing sequence of integers $(\pi_1,\pi_2,\ldots,\pi_k)$ where $\pi_1+\pi_2+\cdots+\pi_k=n$. An overpartition of $n$ is a partition of $n$ where the first appearance of a part of any size may be overlined. For example, $(7,6,\overline{5},5,5,\overline{3},2,2,2)$ is an overpartition of $37$. We can create a graphical representation of a partition, called a Ferrers diagram, by making an array of boxes whose $i$th row has as many boxes as the $i$th part of the partition.  There is a variation of a Ferrers diagram called an $m$-modular diagram (also called a MacMahon diagram) where the part $mj+r$ with $0\le r<m$ is represented by a row made of $j$ boxes containing an $m$ following one box containing a $r$.  

For this paper, we create an analogue of a $2$-modular diagram called a boxed $2$-modular diagram, which is a graphical representation of a pair $(k,\pi)$ where $k$ is a non-negative integer and $\pi$ is a partition. To obtain the boxed $2$-modular diagram, we represent $\pi$ as a $2$-modular diagram and $k$ as a row of  one 0 and $k$ 1s at the top of the diagram.  For example, the figure below is a boxed $2$-modular diagram for the pair $(3,(8,7,5,5,3)).$
\begin{center}
\ytableausetup{boxsize = 1.2em}
\begin{ytableau}
0 & 1 & 1 & 1 \\
2 & 2 & 2 & 2 \\
1 & 2 & 2 & 2  \\
1 & 2 & 2   \\
1 & 2 & 2\\
1 & 2  \\
\end{ytableau}\\
\end{center}

Note that, if $\pi$ is a partition into odd parts with the largest part no greater than $2k+1$, the boxed $2$-modular diagram will have the shape of a partition and the boxes in the first column will not contain any 2s. 

To represent an overpartition, we shade the last box of any overlined part.  For example, the figure below is a boxed $2$-modular diagram for the pair $(3,(8,\overline{7}, \overline{5},5,3)).$

\begin{center}
\ytableausetup{boxsize = 1.2em}
\begin{ytableau}
0& 1 & 1 & 1 \\
2 & 2 & 2 & 2 \\
1 & 2 & 2 & *(gray)2  \\
1 & 2 & *(gray) 2   \\
1 & 2 & 2\\
1 & 2  \\
\end{ytableau}\\
\end{center}

We use the following notation when discussing pairs $(k,\overline{\pi})$.
\begin{itemize}
\item $\mathcal{P}_n$ is the set of pairs $(k,\overline{\pi})$ where $k$ is a non-negative integer and $\overline{\pi}$ is an overpartition of $n-k$ into odd parts of size no greater than $2k+1$ and with all overlined parts of size no greater than $2k-1$.
\item $\nu(\overline{\pi})$ is the number of parts of the overpartition $\overline{\pi}$.
\item $s(\overline{\pi})$ is the size of the smallest part of $\overline{\pi}$.
\item $\nu_\ell(k,\overline{\pi})$ is the number of parts of size $2k+1$ in $\overline{\pi}$.
\item $\nu_s(\overline{\pi})$ is the number of times the smallest part appears in $\overline{\pi}$.
\end{itemize}

\section{Combinatorial Intepretation}
In this section, we interpret Theorem \ref{fq4} in terms of pairs $(k,\overline{\pi})\in \mathcal{P}_n.$  We count each pair $(k,\overline{\pi})\in \mathcal{P}_n$ with weight $(-1)^{\nu(\overline{\pi})}$.

\begin{theorem}\label{interpretation}
Let $\overline{p}_o(n)$ (resp. $\overline{p}_e(n)$) be the number of pairs $(k,\overline{\pi})$, where $k$ is a non-negative integer and $\overline{\pi}$ is an overpartition of $n-k$ into an odd number (resp. even number) of odd parts of size not exceeding $2k+1$, where all overlined parts must have size $<2k+1$. Then, for $n\ge 0$, 
$$\overline{p}_e(n)-\overline{p}_o(n)=\begin{cases} (-1)^k, & \text{if } n=2k(k+1),\\
0, & \text{otherwise.}\end{cases}$$
\end{theorem}

\begin{theorem}\label{equivalence}
Theorem \ref{fq4} and Theorem \ref{interpretation} are equivalent.
\end{theorem}
\begin{proof}
The equivalence of the right sides is trivial, so we focus on the left sides. Note that $(q;q^2)_k$ generates partitions into distinct odd parts of size no greater than $2k-1$, where each partition into $\nu$ parts has weight $(-1)^\nu$.  Similarly, $\frac{1}{(-q;q^2)_{k+1}}$ generates partitions into odd parts of size no greater than $2k+1$, where each partition into $\nu$ parts has weight $(-1)^\nu.$ Thus, if we let the parts coming from $(q;q^2)_k$ be overlined, we find that $\frac{(q;q^2)_k}{(-q;q^2)_{k+1}}$ generates overpartitions into odd parts of size no greater than $2k+1$, where all overlined parts are no larger than $2k-1$, and each overpartition into $\nu$ parts is counted with weight $(-1)^\nu$. Additionally, $q^k$ generates the integer $k$.  Therefore $$\sum_{k=0}^\infty \frac{(q;q^2)_k q^k}{(-q;q^2)_{k+1}}=\sum_{n=0}^\infty\{\overline{p}_e(n)-\overline{p}_o(n)\}q^n,$$ where $\overline{p}_e(n)$ and $\overline{p}_o(n)$ are as defined in Theorem \ref{interpretation}.
\end{proof}

\section{Proof of the Main Theorem}
We devote this section to proving Theorem \ref{interpretation} combinatorially. To obtain the bijection, we split $\mathcal{P}_n$ into cases.  First, we show that conjugation is a sign-reversing involution on the case where $k+\nu(\overline{\pi})\equiv 1\pmod 2$.  Then, for the case where $k+\nu(\overline{\pi})\equiv 0\pmod 2$, we further divide this subset of $\mathcal{P}_n$ into cases depending on the relative sizes of the last row and the last column of the boxed $2$-modular diagram and introduce variations of conjugation that provide sign-reversing bijections and involutions on these cases.
\subsection{Conjugation}
For an ordinary partition $\pi$, the conjugate partition $\pi'$ is defined to be the partition created by reflecting the Ferrers diagram of $\pi$ about the line $y=-x$. Similarly, for a pair $(k,\pi)$, where $k$ is a non-negative integer and $\pi$ is a partition into odd parts of size $\le 2k+1$, we can reflect our boxed $2$-modular diagram about the line $y=-x$ to get the conjugate pair $(k',\overline{\pi}')$.

\begin{example}
The conjugate of $\left(4,(9,9,7,7,5,5,3)\right)$ is $ (7,(15, 13, 9, 5)).$
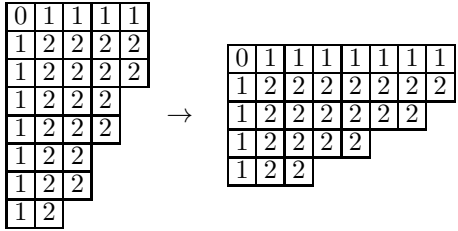
\begin{center}
\ytableausetup{boxsize=1em}
\begin{minipage}{0.17\textwidth}
\begin{ytableau}
0&1&1&1&1\\
1&2&2&2&2\\
1&2&2&2&2\\
1&2&2&2\\
1&2&2&2\\
1&2&2\\
1&2&2\\
1&2
\end{ytableau}
\end{minipage}
$\to$
\;\;
\begin{minipage}{0.35\textwidth}
\begin{ytableau}
0&1&1&1&1&1&1&1\\
1&2&2&2&2&2&2&2\\
1&2&2&2&2&2&2\\
1&2&2&2&2\\
1&2&2
\end{ytableau}
\end{minipage}
\end{center}
\end{example}

Furthermore, if we have a pair $(k,\overline{\pi})\in\mathcal{P}_n$, we can define the conjugate pair $(k',\overline{\pi}')$ by taking the conjugate and overlining the $(j+1)^{\text{st}}$ part in $\overline{\pi}'$ for every $j$ where a part of size $2j+1$ is overlined in $\overline{\pi}$.

\begin{example}
The conjugate of $(3,(7,\overline{5},5,5,\overline{3},1,1,1))$ is $(8,(11,\overline{9},\overline{3})).$
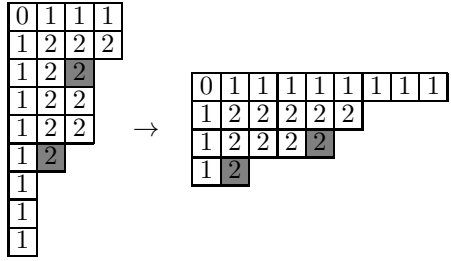
\begin{center}
\begin{minipage}{0.13\textwidth}
\begin{ytableau}
0&1&1&1\\
1&2&2&2\\
1&2&*(gray)2\\
1&2&2\\
1&2&2\\
1&*(gray)2\\
1\\
1\\
1\\
\end{ytableau}
\end{minipage}
$\to$\;\;\;
\begin{minipage}{0.3\textwidth}
\begin{ytableau}
0&1&1&1&1&1&1&1&1\\
1&2&2&2&2&2\\
1&2&2&2&*(gray)2\\
1&*(gray)2
\end{ytableau}
\end{minipage}
\end{center}
\end{example}

Note that, because conjugation swaps rows and columns and preserves the boxes in the diagram, $k'=\nu(\overline{\pi})$, $\nu(\overline{\pi}')=k$, and $k'+|\overline{\pi}'|=k+|\overline{\pi}|$. Furthermore, since conjugation is its own inverse, we obtain the following lemma.

\begin{lemma}\label{conjbij}
Let $\mathcal{S}_{n,k,\ell}$ be the set of pairs $(k,\overline{\pi})$, where $k$ is a non-negative integer and $\overline{\pi}$ is an overpartition of $n-k$ into $\ell$ odd parts of size $\le 2k+1$, with all overlined parts no larger than $2k-1$.  
Then, $|\mathcal{S}_{n,k,\ell}|=|\mathcal{S}_{n,\ell,k}|$.
\end{lemma}

When $k+\nu(\overline{\pi})\equiv 1\pmod2$, conjugation is sign-reversing, which leads to the next lemma.

\begin{lemma}\label{conj}
Conjugation is a sign-reversing involution on pairs $(k,\overline{\pi})\in\mathcal{P}_n$ counted with weight $(-1)^{\nu(\overline{\pi})}$, where $k+\nu(\overline{\pi})\equiv 1\pmod 2$.
\end{lemma}
\begin{proof}
This follows from Lemma \ref{conjbij} and the fact that $k\not\equiv \nu(\overline{\pi})\pmod 2$, so conjugation must be sign-reversing.  
\end{proof}

\subsection{Variations}
For the case $k+\nu(\overline{\pi})\equiv 0\pmod2$, we consider two variations of conjugation. First, we define $\phi_s(k,\overline{\pi})$ by fixing the smallest part of $\overline{\pi}$ and conjugating the remainder of the boxed 2-modular diagram.  
\begin{example}
 $\phi_s(4,(9,9,9,\overline{7},7,\overline{5}))=(5,(11,11,11,\overline{7},\overline{5}))$
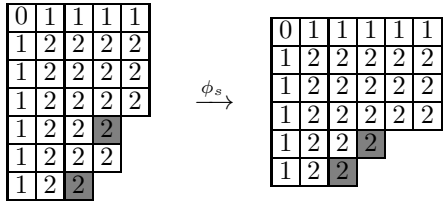
\begin{center}
\begin{minipage}{0.2\textwidth}
\begin{ytableau}
0&1&1&1&1\\
1&2&2&2&2\\
1&2&2&2&2\\
1&2&2&2&2\\
1&2&2&*(gray)2\\
1&2&2&2\\
1&2&*(gray)2\\
\end{ytableau}
\end{minipage} 
$\xrightarrow{\phi_s}$ \quad
\begin{minipage}{0.25\textwidth}
\begin{ytableau}
0&1&1&1&1&1\\
1&2&2&2&2&2\\
1&2&2&2&2&2\\
1&2&2&2&2&2\\
1&2&2&*(gray)2\\
1&2&*(gray)2\\
\end{ytableau}
\end{minipage}
\end{center}
\end{example} 

Note that, $\phi_s$ is well-defined for pairs $(k,\overline{\pi})\in \mathcal{P}_n$ where the last row of the boxed $2$-modular diagram is shorter than the last column. Equivalently, $\phi_s$ is well-defined when $\frac{s(\overline{\pi})-1}{2} < \nu_{\ell}(k,\overline{\pi})$. Furthermore, $\phi_s$ is also well defined when $\frac{s(\overline{\pi})-1}{2}=\nu_{\ell}(k,\overline{\pi})$, $s(\overline{\pi})<2k+1$, and the last part of $\overline{\pi}$ is not overlined.  The last condition is necessary to maintain the restriction on overpartitions that only the first part of any size may be overlined.  If we define $(k_s,\overline{\pi}_s)=\phi_s((k,\overline{\pi}))$, we can note that $k_s=\nu(\overline{\pi})-1$ and $\nu(\overline{\pi}_s)=k+1$. Moreover, the size of the penultimate part of $\overline{\pi}$ determines $\nu_\ell(\phi_s(k,\overline{\pi}))$, so we consider separately the cases where $\nu_s(\overline{\pi})=1$ and $\nu_s(\overline{\pi})>1$. Then, we have the following lemmas.
\begin{lemma}\label{phis}
The map $\phi_s$ is a sign-reversing involution on the set $\{(k,\overline{\pi})\in\mathcal{P}_n: k+\nu(\overline{\pi})\equiv 0\pmod2, s(\overline{\pi})<2\nu_\ell(k,\overline{\pi})+1, \text{ and } \nu_s(\overline{\pi})=1 \}.$
\end{lemma}
\begin{proof}
Let $(k,\overline{\pi})\in \mathcal{P}_n$ such that $k+\nu(\overline{\pi})\equiv 0\pmod2, s(\overline{\pi})<2\nu_\ell(k,\overline{\pi})+1, \text{ and } \nu_s(\overline{\pi})=1.$ Let $(k_s,\overline{\pi}_s)=\phi_s(k,\overline{\pi})$. Since $\nu_s(\overline{\pi})=1$, the second smallest part of $\overline{\pi}$, which determines $\nu_\ell(k_s,\overline{\pi}_s)$, will be larger than $s(\overline{\pi})=s(\overline{\pi}_s)$, so $s(\overline{\pi}_s)<2\nu_\ell(k_s,\overline{\pi}_s)+1$. Moreover, since $s(\overline{\pi})<2\nu_\ell(k,\overline{\pi})+1$, $\nu_s(\overline{\pi}_s)=1$.  Thus, $(k_s,\overline{\pi}_s)\in \mathcal{P}_n$ such that $k_s+\nu(\overline{\pi}_s)\equiv 0\pmod2, s(\overline{\pi}_s)<2\nu_\ell(k_s,\overline{\pi}_s)+1, \text{ and } \nu_s(\overline{\pi}_s)=1$. Finally, since $ \nu(\overline{\pi})\equiv k \not\equiv k+1\pmod2$ and $\nu(\overline{\pi}_s)=k+1$, $\nu(\overline{\pi}_s)\not\equiv \nu(\overline{\pi}) \pmod2$, so the map is sign-reversing. 
\end{proof}

\begin{lemma}\label{phis2}
The map $\phi_s$ is a sign-reversing involution on the set $\{(k,\overline{\pi})\in\mathcal{P}_n: k+\nu(\overline{\pi})\equiv 0\pmod2, s(\overline{\pi})=2\nu_\ell(k,\overline{\pi})+1, s(\overline{\pi})\neq 2k+1 \text{, and } \nu_s(\overline{\pi})>1 \}.$
\end{lemma}

\begin{proof}
Let $(k,\overline{\pi})\in \mathcal{P}_n$ such that $k+\nu(\overline{\pi})\equiv 0\pmod2,\, s(\overline{\pi})=2\nu_\ell(k,\overline{\pi})+1,\, s(\overline{\pi})\neq 2k+1 \text{ and } \nu_s(\overline{\pi})>1.$ Since $s(\overline{\pi})\neq 2k+1$, $\phi_s$ is well-defined and we can let $(k_s,\overline{\pi}_s)=\phi_s(k,\overline{\pi})$. Since $\nu_s(\overline{\pi})>1$, $2\nu_\ell(k_s,\overline{\pi}_s)+1=s(\overline{\pi})=s(\overline{\pi}_s)$.  Furthermore, because $s(\overline{\pi})\ne 2k+1$ and $\nu_s(\overline{\pi})>1$, $\nu_\ell(k,\overline{\pi})<\nu(\overline{\pi})-1$, so $s(\overline{\pi}_s)=s(\overline{\pi})=2\nu_\ell(k,\overline{\pi})+1<2\nu(\overline{\pi})-1=2k_s+1$.  Moreover, since $s(\overline{\pi})=2\nu_\ell(k,\overline{\pi})+1$, $\nu_s(\overline{\pi}_s)>1$.  Thus, $(k_s,\overline{\pi}_s)\in \mathcal{P}_n$ such that $k_s+\nu(\overline{\pi}_s)\equiv 0\pmod2,$ $s(\overline{\pi}_s)=2\nu_\ell(k_s,\overline{\pi}_s)+1,$ $s(\overline{\pi}_s)\ne 2k+1, \text{ and } \nu_s(\overline{\pi}_s)>1$. Finally, as explained above, the map is sign-reversing because $\nu(\overline{\pi}_s)\not\equiv \nu(\overline{\pi}) \pmod 2$. 
\end{proof}

Another variation of conjugation is $\phi_r$, defined as $\phi_r(k,\overline{\pi})=\text{conj}\circ \phi_s\circ \text{conj}(k,\overline{\pi})$, where $\text{conj}$ is the conjugation map described in Section 4.1. Note that this is the same as fixing the right-most column of the boxed $2$-modular diagram, conjugating the remainder, and making a small adjustment to which parts are overlined.
\begin{example}
We have $\phi_r(5,(11,11,9,9,\overline{7},7,7))=(8,(17,17,15,\overline{9})).$
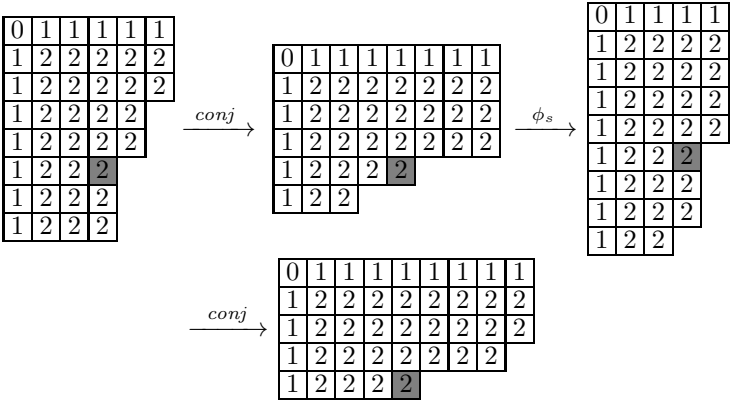
\begin{center}
\begin{minipage}{0.19\textwidth}
\begin{ytableau}
0&1&1&1&1&1\\
1&2&2&2&2&2\\
1&2&2&2&2&2\\
1&2&2&2&2\\
1&2&2&2&2\\
1&2&2&*(gray)2\\
1&2&2&2\\
1&2&2&2
\end{ytableau}
\end{minipage}
$\xrightarrow{\;\text{conj}\;\;}\;$
\begin{minipage}{0.26\textwidth}
\begin{ytableau}
0&1&1&1&1&1&1&1\\
1&2&2&2&2&2&2&2\\
1&2&2&2&2&2&2&2\\
1&2&2&2&2&2&2&2\\
1&2&2&2&*(gray)2\\
1&2&2
\end{ytableau}
\end{minipage}
$\xrightarrow{ \;\;  \phi_s  \;\; }$
\begin{minipage}{0.16\textwidth}
\begin{ytableau}
0&1&1&1&1\\
1&2&2&2&2\\
1&2&2&2&2\\
1&2&2&2&2\\
1&2&2&2&2\\
1&2&2&*(gray)2\\
1&2&2&2\\
1&2&2&2\\
1&2&2\\
\end{ytableau}
\end{minipage}\\
$\xrightarrow{\;\;\text{conj}\;\;}$
\begin{minipage}{0.3\textwidth}
\begin{ytableau}
0&1&1&1&1&1&1&1&1\\
1&2&2&2&2&2&2&2&2\\
1&2&2&2&2&2&2&2&2\\
1&2&2&2&2&2&2&2\\
1&2&2&2&*(gray)2
\end{ytableau}
\end{minipage}
\end{center}
\end{example}

Note that $\phi_r$ is well-defined for pairs $(k,\overline{\pi})\in \mathcal{P}_n$ where the last row of the boxed $2$ modular diagram is longer than the last column.  Equivalently, $\phi_r$ is well-defined when $\frac{s(\overline{\pi})-1}{2}>\nu_\ell(k,\overline{\pi})$. Thus, we obtain the following lemma.
\begin{lemma}\label{phir}
The map $\phi_r$ is a sign reversing involution on the set $\{(k,\overline{\pi})\in\mathcal{P}_n: k+\nu(\overline{\pi})\equiv 0\pmod2, s(\overline{\pi})>2\nu_\ell(k,\overline{\pi})+1, \text{ and } \overline{\pi} \text{ has a part of size } 2k-1 \}.$
\end{lemma}

\begin{proof}
Let $(k,\overline{\pi})\in \mathcal{P}_n$ such that $k+\nu(\overline{\pi})\equiv 0\pmod2, s(\overline{\pi})>2\nu_\ell(k,\overline{\pi})+1, \text{ and } \overline{\pi} \text{ has a part of size } 2k-1.$ Let $(k',\overline{\pi}')$ be the conjugate of $(k,\overline{\pi})$.  Then, $s(\overline{\pi}')<2\nu_\ell(k',\overline{\pi}')+1$ and $\nu_s(\overline{\pi})=1$, so we can apply Lemma 4.3.  
\end{proof}

After applying Lemmas \ref{conj}, \ref{phis}, \ref{phis2}, and \ref{phir}, we are left with four cases, all of which have $k+\nu(\overline{\pi})\equiv 0\pmod 2$.
\begin{itemize}
\item Case 1: Smallest part of $\overline{\pi}$ appears once and is equal to $2\nu_\ell(k,\overline{\pi})+1\ne 2k+1$.
\item Case 2: Smallest part of $\overline{\pi}$ appears multiple times and is smaller than $2\nu_\ell(k,\overline{\pi}) +1$.
\item Case 3: Smallest part of $\overline{\pi}$ is greater than $2\nu_\ell(k,\overline{\pi})+1$ and $\overline{\pi}$ has no part of size $2k-1$.
\item Case 4: $s(\overline{\pi})=2\nu_\ell(k,\overline{\pi})+1=2k+1.$
\end{itemize}

 Note that applying $\phi_s$ to a pair $(k,\overline{\pi})$ in Case 1 reduces the number of distinct parts by one. Since the number of overpartitions of a given shape depends on the number of distinct parts, reducing the number of distinct parts by one requires us to restrict which parts of $\overline{\pi}$ may be overlined.  The next two lemmas provide the details of dividing Case 1 into two halves by considering whether or not the smallest part is overlined. 

\begin{lemma}\label{cases13}
There is a sign-reversing bijection between the pairs in Case 1 where the smallest part is not overlined and the pairs in Case 2.
\end{lemma}
\begin{proof}
Let $(k,\overline{\pi})$ be a pair in Case 1 where the smallest part of $\overline{\pi}$ not overlined. Since the smallest part of $\overline{\pi}$ is not overlined, $\phi_s$ is well-defined. Thus, let $(k_s,\overline{\pi}_s)=\phi_s(k,\overline{\pi})$. Since $\nu_s(\overline{\pi})=1$, $s(\overline{\pi}_s)<2\nu_\ell(k_s,\overline{\pi}_s)+1$.  Moreover, because $s(\overline{\pi})=2\nu_\ell(k,\overline{\pi})+1$, $\nu_s(\overline{\pi}_s)>1$. Therefore, $(k_s,\overline{\pi}_s)$ is in Case 2.    

Since $\phi_s$ is its own inverse, we can take a pair $(k_2,\overline{\pi}_2)$ in Case 2 and apply $\phi_s$ to find a pair in Case 1. Because $\overline{\pi}_2$ has more than one appearance of the smallest part, the last part will not be overlined, so the smallest part of $\phi_s(k_2,\overline{\pi}_2)$ will not be overlined.   
\end{proof}

\begin{lemma}
There is a sign-reversing bijection between the pairs in Case 1, where the smallest part is overlined, and the pairs in Case 3. 
\end{lemma}
\begin{proof}
Note that conjugation is a sign-preserving bijection between the pairs in Case 2 and the pairs in Case 3. Thus, we can remove the overline on the smallest part of $\overline{\pi}$, apply $\phi_s$, and take the conjugate to obtain a sign-reversing bijection between the pairs in Case 1, where the smallest part is overlined, and the pairs in Case 3. 
\end{proof}

Now, the only pairs left are those in Case 4.  These occur exactly when $n=k+|\overline{\pi}|=2k(k+1)$, proving Theorem \ref{interpretation}.

\section{Generalizations}
First, we note that all of our maps preserve the number of boxes containing a 1 in our diagrams.  Furthermore, this number of 1s is exactly $k+\nu(\overline{\pi})$.  Thus, if we let $z$ count the number of 1s in the diagram, we obtain the generalization. 
\begin{theorem}\label{generalization}
$$\sum_{n=0}^\infty \frac{(zq;q^2)_n z^nq^n}{(-zq;q^2)_{n+1}}=\sum_{n=0}^\infty (-1)^nz^{2n}q^{2n(n+1)}.$$
\end{theorem}
Theorem \ref{fq4} is the case $z=1$ of this generalization. Additionally, we can generalize boxed 2-modular diagrams as boxed $m$-modular diagrams by replacing the 2s in the diagram with $m$'s and all 1s with $r$'s to allow parts of size $r\pmod m$ for some fixed $0\le r<m$.  Then, we obtain the following generalization. 

\begin{theorem}\label{rmodm}
$$\sum_{n=0}^\infty \frac{(zq^r;q^m)_n z^nq^{rn}}{(-zq^r;q^m)_{n+1}}=\sum_{n=0}^\infty (-1)^nz^{2n}q^{n(mn+2r)},$$
\end{theorem}
Theorem \ref{rmodm} yields Theorem \ref{fq4} when $z=1$, $m=2$, and $r=1$.

\section{Further work}
This work allows us to obtain a similar combinatorial interpretation for another one of Ramanujan's identities.  
\begin{theorem}[Ramanujan]\label{fq3}
If $f(q)$ is defined by (\ref{fq}), then for $|q|<1$
$$ \sum_{k=0}^\infty \frac{q^n(q;q^2)_{n}}{(-q;q)_{2n+1}}=f(q^3)$$
\end{theorem}

We can interpret Theorem \ref{fq3} in terms of pairs $(k,\overline{\pi})\in \mathcal{P}'_n$ where $\mathcal{P}'_n$ contains all pairs $(k,\overline{\pi})$ where $k\in \mathbb{Z}_{\ge 0}$, $\overline{\pi}$ is an overpartition into parts of size $\le 2k+1$ where all overlined parts are odd and of size $\le 2k-1$, and $k+|\overline{\pi}|=n$. We count each pair with weight $(-1)^{\nu(\overline{\pi})}$.

\begin{theorem}\label{fq3interp}
Let $\overline{p}'_0(n)$ (resp. $\overline{p}'_e(n)$) be the number of pairs $(k,\overline{\pi})\in \mathcal{P}'_n$ where $\overline{\pi}$ has an odd number (resp. even number) of parts. Then, for $n\ge 0$, $$\overline{p}'_e(n)-\overline{p}'_o(n)=\begin{cases} (-1)^k & \text{if } n=\frac{3k(k+1)}{2}\\
0 & \text{otherwise.} 
\end{cases}$$
\end{theorem}
Due to the presence of even parts in the partition, a bijective proof of Theorem \ref{fq3interp} appears to be more difficult than the proof of Theorem \ref{interpretation} and would be a welcome contribution.  We suspect that the involution necessary for a bijective proof of Theorem \ref{fq3interp} will fix pairs $(k, (2k+2k-1+\ldots+(k+1)) )$.

\section*{Acknowledgement}
The author thanks Bruce Berndt for suggesting this project.  Additionally, thanks to Frank Garvan for suggesting Theorem \ref{generalization} and Dennis Eichhorn for his many helpful comments.
\bibliographystyle{hunsrt}

\bibliography{Bib}

\begin{thebibliography}{10}

\bibitem{Rogers1917}
L.~J. Rogers.
\newblock On two theorems of combinatory analysis and some allied identities.
\newblock {\em Proc. London Math. Soc.}, s2-16:315--336, 1917.

\bibitem{RamanujanNB}
S.~Ramanujan.
\newblock {\em {Notebooks (2 volumes)}}.
\newblock Tata Institute of Fundamental Research, Bombay, 1957.

\bibitem{RamanujanLNB}
S.~Ramanujan.
\newblock {\em {The Lost Notebook and Other Unpublished Papers}}.
\newblock Narosa Pub. House, New Delhi, 1988.

\bibitem{Andrews1979}
G.~E. Andrews.
\newblock {\em Partitions: yesterday and today.}
\newblock New Zealand Mathematical Society, Wellington, 1979.

\bibitem{Alladi2009}
K.~Alladi.
\newblock {A partial theta identity of Ramanujan and its number-theoretic
  interpretation}.
\newblock {\em Ramanujan J.}, 20(3):329--339, 2009.

\bibitem{Alladi2010}
K.~Alladi.
\newblock {A combinatorial study and comparison of partial theta identities of
  Andrews and Ramanujan}.
\newblock {\em Ramanujan J.}, 23(1):227--241, 2010.

\bibitem{Berndt2010}
B.~C. Berndt, B.~Kim, and A.~J. Yee.
\newblock {Ramanujan's lost notebook: Combinatorial proofs of identities
  associated with Heine's transformation or partial theta functions}.
\newblock {\em J. Comb. Theory. Ser. A}, 117(7):957--973, 2010.

\bibitem{Berndt2003}
B.~C. Berndt and A.~J. Yee.
\newblock {Combinatorial proofs of identities in Ramanujan's lost notebook
  associated with the Rogers-Fine identity and false theta functions}.
\newblock {\em Ann. Comb.}, 7(4):409--423, 2003.

\bibitem{Kim2010}
B.~Kim.
\newblock {Combinatorial proofs of certain identities involving partial theta
  functions.}
\newblock {\em Int. J. Number Theory}, 6(2):449--460, 2010.

\bibitem{Andrews2005}
G.~E. Andrews and B.~C. Berndt.
\newblock {\em {Ramanujan's Lost Notebook, Part I}}.
\newblock Springer, New York, 2005.

\bibitem{Andrews1981}
G.~E. Andrews.
\newblock {Ramanujan's “lost” notebook. I. partial $\theta$-functions}.
\newblock {\em Adv. Math.}, 41:137--172, August 1981.

\bibitem{Andrews2007}
G.~E. Andrews and S.~O. Warnaar.
\newblock {The Bailey transform and false theta functions}.
\newblock {\em Ramanujan J.}, 14(1):173--188, 2007.

\bibitem{Chu2010}
W.~Chu and W.~Zhang.
\newblock {Bilateral Bailey lemma and false theta functions}.
\newblock {\em Int. J. Number Theory}, 6(3):515--577, 2010.

\bibitem{Wang2018}
L.~{Wang}.
\newblock {New proofs of Ramanujan's identities on false theta functions}.
\newblock {\em ArXiv e-prints}, January 2018, 1801.07956.

\end{thebibliography}

\end{document}